\documentclass[12pt]{amsart}
\usepackage{amsthm,amssymb}
\usepackage{mathtools}
\usepackage{tabmac}
\usepackage[colorlinks]{hyperref}
\usepackage[margin=1in]{geometry}
\usepackage{tikz}
\newcommand{\oX}{\mathring{X}}
\newcommand{\oPi}{\mathring{\Pi}}

\def\spn{{\rm span}}
\def\GL{{\rm GL}}
\def\bC{{\mathbf C}}
\def\A{{\mathcal{A}}}
\def\H{{H}}
\def\M{{\mathcal{M}}}
\def\B{{\mathcal{B}}}
\def\J{{\mathcal{J}}}
\def\row{{\rm row}}
\def\wt{{\rm wt}}

\def\sl{{\mathfrak {sl}}}
\def\gl{{\mathfrak {gl}}}
\def\CC{{\mathcal C}}
\def\C{{\mathbb C}}
\def\R{{\mathbb R}}
\def\Z{{\mathbb Z}}

\def\O{{\mathcal{O}}}
\def\Gr{{\rm Gr}}
\def\hGr{{\rm  \hat Gr}}
\def\hPi{{\hat \Pi}}

\def\X{X}

\def\diag{{\rm diag}}
\def\Spec{{\rm Spec}}

\def\codim{{\rm codim}}
\def\L{{\mathcal{L}}}
\def\A{{\mathcal{A}}}
\def\End{{\rm End}}

\def\te{{\tilde e}}
\def\tf{{\tilde f}}
\def\I{{\mathcal{I}}}
\def\wt{{\rm wt}}

\def\Bound{\B}
\def\P{{\mathbb P}}

\newcommand\defn[1]{{\bf #1}}

\newtheorem{theorem}{Theorem}
\newtheorem{remark}[theorem]{Remark}
\newtheorem{definition}[theorem]{Definition}
\newtheorem{lemma}[theorem]{Lemma}
\newtheorem{proposition}[theorem]{Proposition}
\newtheorem{prop}[theorem]{Proposition}

\newtheorem{cor}[theorem]{Corollary}
\newtheorem{problem}[theorem]{Problem}

\newtheorem{example}[theorem]{Example}
\numberwithin{theorem}{section}

\begin{document}
\title{Cyclic Demazure modules and positroid varieties}
\author{Thomas Lam}
\address{Department of Mathematics, University of Michigan,
2074 East Hall, 530 Church Street, Ann Arbor, MI 48109-1043, USA}
\email{tfylam@umich.edu}\thanks{T.L. was supported by NSF grants DMS-1160726 and DMS-1464693, and by a Von Neumann fellowship at the Institute of Advanced Study.}
\maketitle
\begin{abstract}
A positroid variety is an intersection of cyclically rotated Grassmannian Schubert varieties.  Each graded piece of the homogeneous coordinate ring of a positroid variety is the intersection of cyclically rotated (rectangular) Demazure modules, which we call the cyclic Demazure module.  In this note, we show that the cyclic Demazure module has a canonical basis, and define the cyclic Demazure crystal.
\end{abstract}

\section{Introduction}
The classical Borel-Weil theorem identifies the global sections $\Gamma(G/B,L_\lambda)$ of a line bundle on a flag variety with the irreducible highest weight representation $V(\lambda)$.  When the same line bundle is restricted to a Schubert variety $X_w$, the global sections $\Gamma(X_w, L_\lambda)$ can be identified with the Demazure module $V_w(\lambda)$.  In this paper, we study the global sections $\Gamma(\Pi_f, \O(d))$ of a line bundle on a positroid subvariety $\Pi_f$ of the Grassmannian $\Gr(k,n)$.

Positroid varieties (see Section \ref{sec:Schubpos}) are certain intersections of cyclically rotated Schubert varieties in the Grassmannian.  They were introduced in Postnikov's work \cite{Pos} on the totally nonnegative Grassmannian, and subsequently studied in algebro-geometric terms by Knutson-Lam-Speyer \cite{KLS}.  Via \cite{KLS}, the work of Lakshmibai and Littelmann \cite{LaLi} gives a description of the vector space $\Gamma(\Pi_f, \O(d))$ in terms of standard monomials.  In the present work, we give a new description of $\Gamma(\Pi_f, \O(d))$ that is compatible with the cyclic symmetry of the Grassmannian and its positroid varieties.

We define in Section \ref{sec:cyclicDem} the cyclic Demazure module $V_f(d\omega_k)$ as the intersection of cyclically rotated Demazure modules.  We show in Theorem \ref{thm:Posideal} that a graded piece of the homogeneous coordinate ring of a positroid variety can be identified with the (dual of the) cyclic Demazure module.  

Our approach is based on the key observation (Theorem \ref{thm:canbasis}(4)) that the dual canonical basis of the Grassmannian is invariant under signed cyclic rotation.  This relies heavily on the work of Rhoades \cite{Rho}.  In Theorem \ref{thm:canbasis}(2), we show that dual canonical basis in degree two is identical to the Temperley-Lieb invariants of \cite{LamTL}, which are defined in a combinatorially explicit manner.

We define the cyclic Demazure crystal $B_f(d\omega_k)$ as the intersection of cyclically rotated (via promotion) Demazure crystals.  We show in Theorem \ref{thm:canspan} that $V_f(d\omega_k)$ has a basis given by the canonical basis elements indexed by $B_f(d\omega_k)$.  We obtain the following dichotomy (Theorem \ref{thm:poscan}): a dual canonical basis element either (1) vanishes on $\Pi_f$ (if it is outside $B_f(d\omega_k)$), or (2) it takes strictly positive values on the totally positive part $\Pi_{f,>0}$ (if it is inside $B_f(d\omega_k)$).

Our work was initially motivated by the budding theory of Grassmann polytopes, and many of the results here were announced initially in \cite{LamCDM}.  However, the results herein were so simple and clean, we felt that they deserved a short and separate exposition.  We plan to pursue our intended applications in other work.  In Section \ref{sec:future}, we also indicate some further directions of study.
 
 \medskip
\noindent
{\bf Acknowledgements.} We are grateful to Nima Arkani-Hamed, Allen Knutson, Alex Postnikov, Mark Shimozono, and David Speyer for conversations over the years related to this work.

\section{The dual canonical basis of the Grassmannian}
Let $[n]:= \{1,2,\ldots,n\}$ and $\tbinom{[n]}{k}$ denote the collection of $k$-element subsets of $[n]$.
\subsection{The Grassmannian and its homogeneous coordinate ring}
Let $\Gr(k,n)$ denote the Grassmannian of $k$-planes in $\C^n$ and let $\hGr(k,n)$ denote the affine cone over the Grassmannian.  A point $X \in \hGr(k,n)$ is determined by a set $\Delta_I(X)$ of Pl\"ucker coordinates satisfying the Pl\"ucker relations, where $I \in \tbinom{[n]}{k}$ (see \cite[Section 3]{LamCDM}).  We allow the possibility that all $\Delta_I(X)$ are simultaneously zero.  The Grassmannian $\Gr(k,n)$ is the quotient of $\hGr(k,n) - \{0\}$ by the equivalence relation of simultaneously scaling all Pl\"ucker coordinates by the same scalar.

Let $\hGr(k,n)_{\geq 0}$ denote the cone over the totally nonnegative Grassmannian: it consists of points $X \in \hGr(k,n)_{\geq 0}$ where $\Delta_I(X) \geq 0$ for all $I$.  The totally nonnegative Grassmannian $\Gr(k,n)_{\geq 0}$ \cite{Pos} is the image of $\hGr(k,n)_{\geq 0}$ in $\Gr(k,n)$.

Let $R(k,n)$ denote the coordinate ring of $\hGr(k,n)$, or equivalently, the homogeneous coordinate ring of $\Gr(k,n)$.  Thus,
$$
R(k,n) = \C[\Delta_I]/(\text{Pl\"ucker relations})
$$
is a graded ring where the degree of $\Delta_I$ is taken to be 1.  For example,
$$
R(2,4) = \C[\Delta_{12},\Delta_{13},\Delta_{14},\Delta_{23},\Delta_{24},\Delta_{34}]/(\Delta_{13}\Delta_{24} -\Delta_{12}\Delta_{34} - \Delta_{14}\Delta_{23}).
$$
We also note that $R(k,n)$ is a unique factorization domain.  We let $R(k,n)_d$ denote the $d$-th graded piece of $R(k,n)$, spanned by monomials $\Delta_{I_1} \Delta_{I_2} \cdots \Delta_{I_d}$.  We begin by reviewing the classical description of $R(k,n)_d$ in representation theoretic terms.

\subsection{Highest weight representations}
A partition $\lambda = (\lambda_1 \geq \lambda_2 \geq \cdots \geq \lambda_\ell > 0)$ is a weakly decreasing sequence of positive integers.  We say that $\lambda = (\lambda_1 \geq \lambda_2 \geq \cdots \geq \lambda_\ell > 0)$ has $\ell$ parts and size $|\lambda| = \lambda_1 + \lambda_2 + \cdots + \lambda_\ell$.  We have the following dominance order on partitions: $\lambda \geq \mu$ if and only if $|\lambda| = |\mu|$ and $\lambda_1 \geq \mu_1$, $\lambda_1 + \lambda_2 \geq \mu_1+\mu_2$, and so on.

For a partition $\lambda$ with at most $n$ parts, we have an irreducible, finite-dimensional representation $V(\lambda)$ of $\GL(n)$ with highest weight $\lambda$.  We state some basic facts concerning $V(\lambda)$.

The Young diagram of $\lambda$ is the collection of boxes in the plane with $\lambda_1$ boxes in the 1st row, $\lambda_2$ boxes in the 2nd row, and so on, where all boxes are upper-left justified.  A semistandard tableaux $T$ of shape $\lambda$ is a filling of the Young diagram of $\lambda$ by the numbers $1,2,\ldots,n$ so that each row is weakly-increasing, and each column is strictly increasing.  The weight $\wt(T)$ of a tableau $T$ is the composition $(\alpha_1,\alpha_2,\ldots,\alpha_n)$ where $\alpha_i$ is equal to the number of $i$-s in $T$.  For example, 
$$
\tableau[sbY]{
1&1&3&4&4 \\
2&3&4&5 \\
4&4
}
$$
is a semistandard tableau with shape $(5,4,2)$ with weight $(2,1,2,5,1)$.  Let $B(\lambda)$ denote the set of semistandard tableaux of shape $\lambda$.  (Note that this set depends on $n$, which is suppressed from the notation.)  The dimension $\dim(V(\lambda))$ is equal to the cardinality of $B(\lambda)$.  A vector $v$ in a $\GL(n)$-representation $V$ is called a weight vector with weight $\alpha = (\alpha_1,\alpha_2,\ldots,\alpha_n)$ if the diagonal matrix $\diag(x_1,x_2,\ldots,x_n)$ sends $v$ to $(x_1^{\alpha_1} x_2^{\alpha_2} \cdots x_n^{\alpha_n})v$.

Lusztig \cite{Lus} and Kashiwara \cite{KasGl} have constructed a \defn{canonical basis}, or \defn{global basis} of the $U_q(\sl_n)$-module $V_q(\lambda)$, which is a quantization of $V(\lambda)$.  We shall only use the evaluation of this basis at $q=1$.  After picking a highest weight vector $v_+$ for $V(\lambda)$, 

\begin{quote}
there exists a distinguished basis $\{G(T) \mid T\in B(\lambda)\}$ of $V(\lambda)$ such that each $G(T)$ is a weight vector with weight $\wt(T)$.
\end{quote}

We shall also let $\{\H(T) \mid T \in B(\lambda)\}$ denote the dual basis of $V(\lambda)^*$, called the dual canonical basis.  Let $(\cdot,\cdot)$ denote the unique nondegenerate symmetric bilinear form on $V(\lambda)$ satisfying $(v_+,v_+) = 1$, and $(x \cdot v,u) = (v, x^T \cdot u)$ where $x \in \gl_n$ and $x^T$ denotes the transpose.  We may identify $V(\lambda)$ with $V(\lambda)^*$ via $(\cdot,\cdot)$, and $\H(T)$ becomes a basis of $V(\lambda)$.

\subsection{Crystals}
The set $B(\lambda)$ has the structure of a \defn{crystal graph}.  We will only need the operations
$$
\te_i: B(\lambda) \to B(\lambda) \cup \{0\}
$$
and
$$
\tf_i: B(\lambda) \to B(\lambda) \cup \{0\}
$$
for $i = 1,2,\ldots,n-1$.  Let $T \in B(\lambda)$.  The rowword $\row(T)$ of $T$ is obtained by reading the rows of $T$ from left to right, starting from the bottom row.  

For a fixed $i \in \{1,2,\ldots,n-1\}$, we think of each occurrence of $i$ in $\row(T)$ to be a closed parentheses ``)" and each occurrence of $i+1$ in $\row(T)$ to be an open parentheses ``(".  We then pair these parentheses as usual until no pairing can be done.  We are left with a sequence that looks like ``))))((".  The operation $\tf_i$ changes the $i$ corresponding to the rightmost unpaired ``)" into a $i+1$.  The operation $\te_i$ changes the $i+1$ corresponding to the leftmost unpaired ``(" into a $i$.  The result will be the rowword of a unique tableau $\te_i(T)$ or $\tf_i(T)$ of shape $\lambda$.  If there is no such ``)" (resp. ``(") then $\tf_i(T)$ (resp. $\te_i(T)$) is defined to be 0.

\subsection{Kirillov-Reshetikhin crystals}
Let $\omega_k = (1,1,\ldots, 1)$ be the partition with $k$ $1$'s.  Then $V(\omega_k)$ is isomorphic to the $k$-exterior power $\Lambda^k(\C^n)$ of the standard representation $\C^n$ of $\GL(n)$ and the canonical basis of $V(\omega_k)$ is simply the basis $\{e_{i_1} \wedge e_{i_2} \wedge \cdots \wedge e_{i_k}\}$.  For an integer $d \geq 1$, the representation $V(d\omega_k)$ for a rectangular partition has very special properties.  The set $B(d\omega_k)$ is the set of semistandard Young tableaux with $k$ rows and $d$ columns.  For example,
$$
\tableau[sbY]{
1&1&3&4&4 \\
2&3&4&5&5 \\
4&4&6&6&6
}
$$
belongs to $B(5\omega_3)$.  

The crystal $B(d\omega_k)$ has an additional operation called \defn{promotion}, which is a bijection $\chi:B(d\omega_k) \to B(d\omega_k)$.  We have $\chi^n = 1$.  Promotion is defined as follows: first remove all occurrences of the letter $n$ in $T$.  Then slide the boxes to the bottom right of the rectangle, always keeping the  rows weakly-increasing and columns strictly-increasing.   Once all slides are complete, we add one to all letters, and fill the empty boxes with the letter $1$ to obtain $\chi(T)$.  For example,
$$
\tableau[sbY]{
1&1&3&4&4 \\
2&3&4&5&5 \\
4&4&6&6&6
}
\raisebox{0.6cm}{\;\;$\to$\;\;}
\tableau[sbY]{
1&1&3&4&4 \\
2&3&4&5&5 \\
4&4
}
\raisebox{0.6cm}{\;\;$\to$\;\;}
\tableau[sbY]{
\bl&\bl&\bl&1&1 \\
2&3&3&4&4 \\
4&4&4&5&5
}
\raisebox{0.6cm}{\;\;$\to$\;\;}
\tableau[sbY]{
1&1&1&2&2 \\
3&4&4&5&5 \\
5&5&5&6&6
}
$$
We have $\chi \circ \te_i = \te_{i+1} \circ \chi$ (resp. $\chi \circ \tf_i = \tf_{i+1} \circ \chi$), and this defines extra operations $\te_0$ and $\tf_0$ on $B(d\omega_k)$.  Together these structures form part of the affine crystal structure of $B(d\omega_k)$, which in this case is a \defn{Kirillov-Reshetikhin crystal}.

\subsection{The dual canonical basis of the Grassmannian}
By the classical Borel-Weil theorem, the degree $d$ component $R(k,n)_d$ of the graded ring $R(k,n)$ is canonically isomorphic, as a $\GL(n)$-representation, to the dual $V(d\omega_k)^*$ of the highest weight representation $V(d\omega_k)$.  

Let $\chi: \C^n \to \C^n$ denote the (signed) cyclic rotation linear map given by sending $e_i$ to $e_{i+1}$ for $i = 1,2,\ldots,n-1$ and sending $e_n$ to $(-1)^{k-1} e_1$.  It also induces a rotation map $\chi:\Gr(k,n) \to\Gr(k,n)$.  Since $\chi \in \GL(n)$, we obtain a cyclic rotation map $\chi: V(d\omega_k) \to V(d\omega_k)$.

\begin{theorem}\label{thm:canbasis}
The vector space $R(k,n)_d$ has a dual canonical basis $\{\H(T) \mid T \in B(d\omega_k)\}$ with the following properties:
\begin{enumerate}
\item For $d = 1$, we have $\H(T) = \Delta_I$, where $I$ is the set of entries in the one-column tableau $T$.
\item For $d= 2$, the set $\{\H(T) \mid T \in B(2\omega_k)\}$ is exactly the set of Temperley-Lieb invariants $\{\Delta_{(\tau,T)} \mid (\tau,T) \in \A_{k,n}\}$ of Section \ref{sec:TL}.
\item For any $T \in B(d\omega_k)$, the function $\H(T)$ is a nonnegative function on $\hGr(k,n)_{\geq 0}$.
\item For any $T \in B(d\omega_k)$, we have $\chi^*(\H(T)) = \H(\chi(T))$, where $\chi^*$ is the pullback map induced by $\chi: \Gr(k,n) \to \Gr(k,n)$.
\end{enumerate}
\end{theorem}
Theorem \ref{thm:canbasis}(1) is well-known.
Theorem \ref{thm:canbasis}(2) will be discussed in Section \ref{sec:TL}.  Theorem \ref{thm:canbasis}(3) is due to Lusztig \cite{LusTP}.  We deduce Theorem \ref{thm:canbasis}(4) from a result of Rhoades \cite{Rho} in the next subsection.  

Already for $d = 3$, the canonical basis of $V(3\omega_k)$ is combinatorially obscure to us.  In \cite{LamTL}, we studied the closely related {\it web basis} in combinatorial terms.


\subsection{Cyclicity of canonical basis}

\begin{theorem}\label{thm:cyclic}
We have $\chi(G(T)) = G(\chi(T))$ and $\chi^*(H(T)) = H(\chi(T))$.
\end{theorem}
\begin{proof}
\def\tchi{\tilde \chi}
Let $\tchi$ denote the unsigned cyclic rotation map that sends $e_i$ to $e_{i+1 \mod n}$.  We first show that $\tchi(G(T)) = \pm G(\chi(T))$.

Let $\A(n)$ denote the coordinate ring of $n \times n$ matrices, so that $\A(n) = \C[x_{ij} \mid i,j \in [n]]$.  The ring $\A(n)$ has a dual canonical basis $b_{P,Q}$ labeled by pairs of semistandard tableaux $P,Q$ of the same shape and entries bounded by $n$.  The cyclic rotation $\tchi$ acts on $\A(n)$ by sending the matrix entry $x_{ij}$ to $x_{i+1,j}$ where indices are taken modulo $n$.  Equivalently, thinking of $\A(n)$ as the space of polynomial functions on $\End(\C^n)$, we have
$$
(\tchi \cdot f)(g) = f(\tchi^{-1} g),
$$
for $f \in \A(n)$ and $g \in \End(\C^n)$.  In \cite[Proposition 5.5]{Rho}, Rhoades shows that when $P,Q$ have rectangular shape, we have
\begin{equation}\label{eq:Rho}
\tchi \cdot b_{P,Q} = \pm b_{\chi(P),Q} + \text{other terms}
\end{equation}
where the other terms belong to the span of the dual canonical basis indexed by shapes different to the shape of $P,Q$.  

The basis studied in \cite{Rho} is connected to the canonical bases of the highest weight representations $V(d\omega_k)$ via the works of Skandera \cite{Ska} and Du \cite{Du}.  Specifically, Du shows that the comodule map $\tau: V(d\omega_k) \to V(d\omega_k) \otimes \A(n)$ sends the highest weight vector $v_+$ of $V(d \omega_k)$ to the sum
$$
\tau(v_+) = \sum_{b_{P,Q}\in T(d,k)} G(P) \otimes b_{P,Q}
$$
where $T(d,k)$ is a subset of the dual canonical basis of $\A(n)$ and $G(P)$ belongs to the canonical basis of $V(d\omega_k)$.  (The coincidence of Du's basis with Lusztig's is shown in \cite{Du2}.)
A computation from the definitions shows that $(\tchi \otimes \tchi)\tau(v_+) = \tau(v_+)$.  It follows that
\begin{equation}\label{eq:v+}
\tau(v_+) = \sum_{b_{P,Q}\in T(m,k)} \tchi(G(P)) \otimes \tchi \cdot b_{P,Q}.
\end{equation}
The set $T(d,k)$ consists of all $b_{P,Q}$ where $Q$ is some fixed semistandard tableaux of rectangular shape $d^k$.  It follows that $\tchi \cdot b_{P,Q} \in T(d,k)$, and in \eqref{eq:v+} all the ``other terms'' from \eqref{eq:Rho} cancel out.  We conclude that $\tchi(G(P)) = \pm G(\chi(P))$.

It follows easily from the fact that $G(P)$ is a weight-vector that $\chi(G(P)) = \pm G(\chi(P))$ and by duality we have $\chi^*(\H(P)) = \pm \H(\chi(P))$.  Now $\chi(\Gr(k,n)_{\geq 0}) = \Gr(k,n)_{\geq 0}$, so it follows from Theorem \ref{thm:canbasis}(3) that we must have $\chi^*(\H(P)) = \H(\chi(P))$.  (Note that $\H(P)$ cannot be identically 0 on $\Gr(k,n)_{\geq 0}$ because the latter is Zariski-dense in $\Gr(k,n)$.)
\end{proof}

\section{Temperley-Lieb immanants}\label{sec:TL}
\subsection{Dual canonical basis for $R(k,n)_2$}
A $(k,n)$-partial noncrossing pairing $(\tau,T)$ consists of a noncrossing pairing $\tau$ of a subset $S \subset [n]$ together with a subset $T \subset [n] \setminus S$ of marked vertices, satisfying $2|T|+|S| = k$.  In \cite{LamTL} we constructed a basis $\{\Delta_{(\tau,T)}\}$ for $R(k,n)_2$ labeled by $(k,n)$-partial noncrossing pairings $(\tau,T)$.  Here is an example of a $(4,8)$-partial non-crossing pairing, where the vertex $6$ is marked:
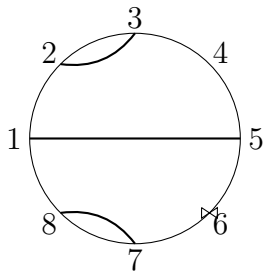
\begin{center}
\begin{tikzpicture}[scale=0.7]
\coordinate (a4) at (45:2);
\coordinate (a3) at (90:2);
\coordinate (a2) at (135:2);
\coordinate (a1) at (180:2);
\coordinate (a8) at (225:2);
\coordinate (a7) at (270:2);
\coordinate (a6) at (315:2);
\coordinate (a5) at (0:2);

\node at (45:2.3) {$4$};
\node at (90:2.3) {$3$};
\node at (135:2.3) {$2$};
\node at (180:2.3) {$1$};
\node at (225:2.3) {$8$};
\node at (270:2.3) {$7$};
\node at (315:2.3) {$6$};
\node at (0:2.3) {$5$};

\draw (0,0) circle (2);

\draw[thick] (a2) to [bend right] (a3);
\draw[thick] (a5) to (a1);
\draw[thick] (a7) to [bend right] (a8);

\begin{scope}[{shift={(a6)}}]
\draw (-0.15,0.1) -- (0.15,-0.1) -- (0.15,0.1) -- (-0.15,-0.1) -- (-0.15,0.1);
\end{scope}

\end{tikzpicture}
\end{center}

Let $(I,J) \in \tbinom{[n]}{k}^2$.  We say that that $(I,J)$ is compatible with $(\tau,T)$ if (a) $I \cap J = T$, and (b) each pair of matched boundary vertices in $\tau$ contains one element of $I$ and one element of $J$.  Let $\CC(I,J)$ denote the set of $(k,n)$-partial noncrossing pairings that are compatible with $(I,J)$.  We have the following identity \cite{LamTL} which determines the elements $\Delta_{(\tau,T)} \in R(k,n)_2$ uniquely.
\begin{theorem}\label{thm:LamTL}  Let $I, J \in \tbinom{[n]}{k}$.  Then
\begin{equation}\label{eq:TL}
\Delta_I \Delta_J = \sum_{(\tau,T) \in \CC(I,J)} \Delta_{(\tau,T)}.
\end{equation}
\end{theorem}

Define a bijection $\theta:\A_{k,n} \to B(2\omega_k)$ as follows.  Given $(\tau, T)$, the tableau $\theta(\tau,T)$ has columns $I_1,I_2$, where $I_1 \cap I_2 = T$, and for each strand $(a,b) \in \tau$ with $a<b$, we have $a \in I_1$ and $b \in I_2$.

\def\K{{\mathcal{K}}}
\def\L{{\mathcal{L}}}

We now prove Theorem \ref{thm:canbasis}(2).
\begin{proposition}
We have $\Delta_{(\tau,T)} = \H(\theta(\tau,T))$.  Thus Temperley-Lieb immanants are the dual canonical basis of $V(2\omega_k)$: 
\end{proposition}
\begin{proof}
We deduce the proposition from setting $q=1$ in work of Brundan \cite{Bru} and Cheng-Wang-Zhang \cite{CWZ}.  In \cite[Section 4]{CWZ}, the dual canonical basis for $V(\omega_k) \otimes V(\omega_k)$ is constructed for any $k$.  The dual canonical basis elements are denoted $\L_f$ in \cite{CWZ}; we shall write them as $\L_{I,J}$ where $I, J$ are $k$-element subsets of $[n]$ (note that \cite{CWZ} are working with $n = \infty$).  The standard basis of $V(\omega_k) \otimes V(\omega_k)$ will be denoted by $\K_{I,J}$.

By \cite[Theorem 26]{Bru}, there is a linear map $\xi: V(\omega_k) \otimes V(\omega_k) \to V(2\omega_k)$ which sends $\L_{A,B}$ to $\H(T)$ if $A,B$ are the two columns of a semistandard tableau $T$ of shape $2^k$, and to $0$ otherwise.  In our notation, the map $\xi$ also sends the standard basis element $\K_{I,J}$ to the monomial $\Delta_I \Delta_J$.  By Theorem \ref{thm:LamTL}, it thus suffices to show that for a 2-column tableau $T$ with columns $A,B$, that the coefficient of $\L_{T}:=\L_{A,B}$ in $\K_{I,J}$ is equal to 1 or 0 depending on whether $\theta^{-1}(T)$ is compatible with $(I,J)$ or not.

Let $A,B$ be two $k$-element subsets of $[n]$.  Cheng-Wang-Zhang \cite{CWZ} define a set of pairs $\Sigma^-_{A,B} = \{(i,j)\}$ (denoted $\Sigma^-_f$ in \cite{CWZ}).  First, let $\A_{A,B}$ be the set of ordered pairs $(i,j)$, where $i \in A \setminus B$, $j \in B \setminus A$, and $i < j$.  Recursively define $\Sigma^{-,r}_{A,B}$ for $r \geq 1$ by
$$
\Sigma^{-,r}_{A,B}:= \left\{(i,j) \in \A_{A,B} \mid j - i = r \text{ and $i, j$ do not appear in }  \bigcup_{1 \leq \ell < r} \Sigma^{-,\ell}_{A,B} \right\}.
$$
We set $\Sigma^-_{A,B}:= \bigcup_{r \geq 1} \Sigma^{-,r}_{A,B}$.  It is then shown in \cite[Corollary 4.18]{CWZ}\footnote{The parameters $m$ and $n$ in \cite{CWZ} are both equal to $k$ for us.} that the coefficient of $\L_{A,B}$ in $\K_{I,J}$ is equal to 1 if $(I,J)$ can be obtained from $(A,B)$ by swapping the pairs in some subset $\Sigma \subseteq \Sigma^-_{A,B}$, and equal to 0 otherwise. 

Now, suppose that $A,B$ are the two columns of a semistandard tableaux $T$ of shape $2^k$.  It is then easy to check that $\Sigma^-_{A,B}$ is exactly the set of strands of the non-crossing matching of $\theta^{-1}(A,B)$.  This completes the proof.
\end{proof}

\subsection{Explicit formula for Temperley-Lieb invariants}
\def\bI {{\bf {I}}}
Call a pair $(I,J)$ {\it standard} if $I$ and $J$ form the two columns of a semistandard tableaux.  Recall that $\theta^{-1}(I,J) = (\tau(I,J),T = I \cap J)$ is a partial noncrossing matching.  


We shall need to consider pairs $(\bI,J)$ where $\bI$ is an {\it ordered} sequence $(i_1,i_2,\ldots,i_k)$ of distinct integers in $[n]$.  Let $\bar \bI \in \tbinom{[n]}{k}$ denote the $k$-element subset consisting of the elements of the sequence $\bI$. We say that $(\bI,J)$ is standard if $(\bar \bI,J)$ is.  We also have a matching $\tau(\bI,J):= \tau(\bar\bI,J)$.  We write $(\bI = (i_1,\ldots,i_a,\ldots,i_k), J) \to_a (\bI' = (i_1,\ldots,j,\ldots,i_k),J' = J \setminus \{j\} \cup \{i_a\})$ if $(i_a,j) \in \tau(\bI,J)$ and $(\bI',J')$ is a standard pair.  

A {\it legal path} $P$ of length $|P|=r$ between $(I,J)$ and $(K,L)$ is a sequence
$$
(\bI_0,J) \to_{a_1} (\bI_1,J_1) \to_{a_2} \cdots \to_{a_r} (\bI_r,J_r)
$$
where $\bI_0$ is equal to $I$ arranged in order, $\bar \bI_r = K$, and $a_1 \geq a_2 \geq \cdots \geq a_r$.  Note that $I \cup J = K\cup L$ as multisets whenever a legal path exists.

The following result can be deduced from \cite{CWZ}.  We give an independent proof.

\begin{theorem}\label{thm:inverse}
We have
\begin{equation}\label{eq:inverse}
\Delta_{(\tau,T)} = \sum_{(I,J)} \left(\sum_P (-1)^{|P|} \right) \Delta_I \Delta_J
\end{equation}
where the first summation is over all standard pairs $(I,J)$ and the second summation is over legal paths from $(I,J)$ to $\theta(\tau,T)$.
\end{theorem}

\begin{example}
Let $k = 3$ and $n = 6$.  Then the standard pairs $(I,J)$ with $I \cup J = [6]$ are
$$
(123,456), (124,356), (134,256), (125,346), (135,246).
$$
The transition matrix from $\{\Delta_I \Delta_J\}$ to $\{\Delta_{\theta^{-1}(I,J)}\}$ and its inverse are
$$
\begin{bmatrix}
1&1&0&0&1\\
0&1&1&1&1\\
0&0&1&0&1\\
0&0&0&1&1\\
0&0&0&0&1
\end{bmatrix} 
\qquad
\text{and}
\qquad
\begin{bmatrix}
1&-1&1&1&-2\\
0&1&-1&-1&1\\
0&0&1&0&-1\\
0&0&0&1&-1\\
0&0&0&0&1
\end{bmatrix} 
$$
respectively.  Reading the last column of the right matrix, we get
$$
\Delta_{\theta^{-1}(135,246)} = \Delta_{135}\Delta_{246} - \Delta_{125}\Delta_{346} - \Delta_{134}\Delta_{256} + \Delta_{124}\Delta_{356} - 2 \Delta_{123}\Delta_{456}.
$$
The term $2\Delta_{123}\Delta_{456}$ arises from the two legal paths $$((1,2,3),\{456\}) \to_{2} ((1,5,3),\{246\})$$ $$((1,2,3),\{456\}) \to_3 ((1,2,4), \{356\}) \to_3 ((1,2,5),\{346\}) \to_2 ((1,3,5),\{246\}).$$  Note that the path  $$((1,2,3),\{456\}) \to_3 ((1,2,4), \{356\}) \to_2 ((1,3,4),\{256\}) \to_3 ((1,3,5),\{246\})$$ is not legal, because the sequence $3,2,3$ is not weakly decreasing.
\end{example}

If $(i < j), (r < s) \in \tau$ are two strands of a noncrossing matching, we say that $(i,j)$ is nested under $(r,s)$ if $(r < i < j < s)$.  We say that $(i,j)$ is nested immediately under $(r,s)$ if, in addition, there is no strand $(p,q)$ such that $(i,j)$ is nested under $(p,q)$, and $(p,q)$ is nested under $(r,s)$.

\begin{lemma}\label{lem:nest}
Suppose that $(\bI,J) \to_a (\bI',J')$, where $(i_a,j)$ is the strand swapped.  Then there is a unique $(i_b,j') \in \tau(\bI,J)$ such that $(i_a,j)$ is nested immediately under $(i_b,j')$.  Furthermore, $\tau(\bI',J')$ is obtained from $\tau(\bI,J)$ by replacing the two strands $(i_a,j)$ and $(i_b,j')$ by the two strands $(i_b,i_a)$ and $(j,j')$.
\end{lemma}
\begin{proof}
If $(i_a,j)$ is not nested under any other strand, then swapping $i_a$ with $j$ in $(\bar I, J)$ cannot give a standard pair.  This gives the first statement.  It is easy to see that replacing $(i_a,j)$ and $(i_b,j')$ by the two strands $(i_b,i_a)$ and $(j,j')$ does indeed give a noncrossing matching, and the second statement follows.
\end{proof}

\begin{lemma}\label{lem:legal}
Suppose that we have a legal path ending at $(\bI = (i_1,i_2,\ldots,i_k),J)$.  Suppose that $a < b$ and that $(i_a,j)$ and $(i_b,j')$ are both in $\tau(\bI,J)$.  Then $(i_a,j)$ is never nested under $(i_b,j')$.
\end{lemma}
\begin{proof}
Let $P = (\bI_0,J) \to_{a_1} (\bI_1,J_1) \to_{a_2} \cdots \to_{a_r} (\bI_r,J_r)$ be a legal path ending at $(\bI_r, J_r) = (\bI = (i_1,i_2,\ldots,i_k),J)$, and suppose $a < b$.   We proceed by induction on $r$.  If $r = 0$, the claim is clear.  If $i_a < i_b$, the claim is clear.  Thus we may assume that $i_a > i_b$ and $a_r \leq a$.  If $a_r < a$, then by induction and Lemma \ref{lem:nest}, the last swap $\to_{a_r} (\bI_r,J_r)$ does not affect the strands $(i_a,j)$ and $(i_b,j')$ incident to $i_a$ and $i_b$.  

Finally, suppose that $a_r = a$, and let $(\bI',J') = (\bI_{r-1},J_{r-1})$.  There are two cases: $i'_a < i'_b$ and $i'_a > i'_b$.  In the first case, the claim follows from Lemma \ref{lem:nest}, and in the second case, the claim follows from the inductive assumption and Lemma \ref{lem:nest}.
\end{proof}

\begin{proof}[Proof of Theorem \ref{thm:inverse}]
Any $i \in T$ is present in both $I$ and $J$ for all terms on the RHS.  Thus it suffices to prove the statement assuming that $T = \emptyset$ and $\tau$ is a complete noncrossing matching on $[n] = [2k]$.  Henceforth, we make this assumption; thus we restrict to standard pairs $(I,J)$ using each element in $[2k]$ exactly once.  Restricting Theorem \ref{thm:LamTL} to these standard pairs, we must show that \eqref{eq:TL} and \eqref{eq:inverse} give inverse matrices.  

Define a partial order $\leq$ on standard pairs by $(I,J) \leq (C,D)$ if $i_r \leq c_r$ for $r=1,2,\ldots,k$, where $I = \{i_1 < i_2 < \cdots < i_k\}$ and $C = \{c_1 < c_2 < \cdots < c_k\}$.  A legal path from $(I,J)$ to $(C,D)$ exists only if $(I,J) \leq (C,D)$.  Also $\theta^{-1}(C,D)$ is compatible with $(K,L)$ only if $(C,D) \leq (K,L)$.  The transition matrices from \eqref{eq:TL} and \eqref{eq:inverse} are triangular with respect to this partial order.  Thus let $(I,J) \leq (K,L)$ be given.  We must show that 
\begin{equation}\label{eq:doublesum}
\sum_{(C,D)} \sum_P (-1)^{|P|} = \begin{cases} 1 & \mbox{if $(I,J) = (K,L)$,} \\
0 & \mbox{if $(I,J) < (K,L)$,}
\end{cases} 
\end{equation}
where the first summation is over all standard pairs $(C,D)$ such that $\theta^{-1}(C,D) \in \CC(K,L)$, and the second summation is over all legal paths $P$ from $(I,J)$ to $(C,D)$.  The statement is clear when $(I,J) = (K,L)$.  

Suppose $(I,J) < (K,L)$.  We provide a sign-reversing involution $\iota$ on the terms in \eqref{eq:doublesum}.  If $\tau(C,D)$ is compatible with $(K,L)$ we can obtain $(K,L)$ from $(C,D)$ by swapping some (uniquely determined) subset $S(C,D)$ of the strands in $\tau(C,D)$.  Let a legal path $P = \cdots (\bC'',D'') \to_x (\bC,D)$ from $(I,J)$ to $(C,D)$ be given, where $\bC = (c_1,c_2,\ldots,c_k)$.  With respect to $(\bC,D)$, the {\it minimum strand} $(c_i,d)$ in $S(C,D)$ is the strand where $i$ is minimal.  We define $\iota(P)$ by splitting into two cases.  

\noindent Case (1): If $|S(C,D)| >0$ and $(c_i,d)$ is the minimal strand, and either 
\begin{enumerate}
\item
$P$ is empty, or
\item 
$P$ is nonempty and $i \leq x$, 
\end{enumerate}
then $\iota(P) = P \to_i (\bC',D')$ is the path obtained by concatenating to $P$ the swap $(\bC,D) \to_i (\bC',D')$.  Note that $\tau (\bar \bC',D')$ is still compatible with $(K,L)$: if $(c_i,d)$ is nested under a strand $(c_s,d')$, then by Lemmas \ref{lem:nest} and \ref{lem:legal}, we cannot have $(c_s,d') \in S(C,D)$.  Thus $S(\bar \bC',D') = S(C,D) \setminus (c_i,d)$.

\noindent Case (2): If either 
\begin{enumerate}
\item $|S(C,D)| = 0$ (that is, $(C,D) = (I,J)$), or
\item $|S(C,D)|>0$ with $(c_i,d)$ minimal strand, and $P$ is nonempty and $i > x$,
\end{enumerate}
then $\iota(P)$ is obtained from $P$ by removing the last swap, so that $\iota(P)$ now ends at $(\bC'',D'')$.  By Lemmas \ref{lem:nest} and \ref{lem:legal} again, note that $\tau(\bC'',D'')$ is compatible with $(I,J)$, and we have $S(\bar \bC'',D'') = S(C,D) \cup \{(p,q)\}$ where $(p,q)$ is the last swap in $P$.

Finally, it is straightforward to verify that $\iota$ is an involution and that $(-1)^{|\iota(P)|} = -(-1)^{|P|}$.
\end{proof}

\section{Schubert varieties and positroid varieties}\label{sec:Schubpos}

\subsection{Schubert varieties}\label{sec:Schub}
Let $I \in \tbinom{[n]}{k}$ be a $k$-element subset of $[n]$.  Let $F_\bullet = \{ 0 = F_0 \subset F_1 \subset \cdots F_{n-1} \subset F_n = \C^n\}$ be a flag in $\C^n$, so that $\dim F_i = i$.  
The Schubert cell $\oX_I(F_\bullet)$ is given by
\begin{equation}\label{eq:schub}
\oX_I(F_\bullet) \coloneqq \{X \in \Gr(k,n) \mid \dim(X \cap F_j) =\#(I \cap [n-j+1,n]) 
\text{ for all } j \in [n]\}.
\end{equation}
The Schubert variety $X_I(F_\bullet)$ is given by
\begin{equation}
X_I(F_\bullet) \coloneqq \{X \in \Gr(k,n) \mid \dim(X \cap F_j) \geq \#(I \cap [n-j+1,n]) 
\text{ for all } j \in [n]\}.
\end{equation}
We have $X_I(F_\bullet) = \overline{\oX_I(F_\bullet)}$.  Also, $X_{[k]}(F_\bullet) = \Gr(k,n)$ and $\codim(X_I(F_\bullet)) = i_1+i_2+\cdots+i_k - (1+2+\cdots+k)$, where $I = \{i_1,i_2\ldots,i_k\}$.  Here and elsewhere, we always mean complex (co)dimension when referring to complex subvarieties.

Let $E_\bullet$ be the standard flag defined by $E_i = \spn(e_n,e_{n-1},\ldots,e_{n-i+1})$.  Then we set the standard Schubert varieties to be $X_I \coloneqq X_I(E_\bullet)$.  Suppose $v_1,v_2,\ldots,v_n$ are the columns of a $k\times n$ matrix (with respect to the basis $e_1,e_2,\ldots,e_n$) representing $X \in \Gr(k,n)$.  Then the condition $\dim(X \cap E_j) = d$ is equivalent to the condition $\dim \spn(v_1,\ldots,v_{n-j}) = k-d$.  Thus the Schubert variety $X_I(E_\bullet)$ is cut out by rank conditions on initial sequences of columns of $X$.

\subsection{Bounded affine permutations, Grassmann necklaces, and positroids}

A \defn{$(k,n)$-bounded affine permutation} is a bijection $f: \Z \to \Z$ satisfying conditions:
\begin{enumerate}
\item $f(i+n) = f(i) + n$,
\item $\sum_{i=1}^n f(i) = \binom{n+1}{2}+kn$,
\item
$i \leq f(i) \leq i+n$.
\end{enumerate}
The set $\Bound(k,n)$ of $(k,n)$-bounded affine permutations forms a lower order ideal in the Bruhat order of the affine symmetric group  (\cite{KLS}).  

Let $I =\{i_1<i_2<\cdots<i_k\}$ and $J =\{j_1<j_2<\cdots <j_k\}$ be two $k$-element subsets of $[n]$.  We define a partial order  $\leq$ on $\tbinom{[n]}{k}$ by $I \leq J$ if $i_r \leq j_r$ for $r=1,2,\ldots,k$. We write $\leq_a$ for the cyclically rotated ordering $a <_a a+1 <_a \cdots <_a n <_a 1 <_a \cdots <_a a-1$ on $[n]$.  Replacing $\leq$ by $\leq_a$, we also have the cyclically rotated version partial order $I \leq_a J$ on $\tbinom{[n]}{k}$.

A \defn{$(k,n)$-Grassmann necklace} \cite{Pos} is a collection of $k$-element subsets $\I = (I_1,I_2,\ldots,I_n)$ satisfying the following property: for each $a \in [n]$:
\begin{enumerate}
\item
$I_{a+1} = I_a$ if $a \notin I_a$
\item
$I_{a+1} = I_a -\{a\} \cup \{a'\}$ if $a \in I_a$.
\end{enumerate}
There is a partial order on the set of $(k,n)$-Grassmann necklaces, given by $\I \leq \J$ if $I_a \leq_a J_a$ for all $a =1,2,\ldots,n$.

Given $f \in \Bound(k,n)$, we define a sequence $\I(f) = (I_1,I_2,\ldots,I_n)$ of $k$-element subsets by the formula
$$
I_a = \{f(b) \mid b < a \text{ and } f(b) \geq a\} \mod n
$$
where $\mod n$ means that we take representatives in $[n]$.  For example, let $k =2$, $n = 6$, and $f = [246759]$.  Then $\I(f) = (13,23,34,46,16,16)$.

\begin{prop}\label{prop:Grassorder}
The map $f \mapsto \I(f)$ is a bijection between $(k,n)$-bounded affine permutations and $(k,n)$-Grassmann necklaces.  
\end{prop}

The inverse map $\I \mapsto f(\I)$ is given as follows.  Suppose $a \notin I_a$.  Then define $f(a) = a$.  Suppose $a \in I_a$ and $I_{a+1} = I_a -\{a\} \cup \{a'\}$.  Then define $f(a) = b$ where $b \equiv a' \mod n$ and $a < b \leq a+n$.  

The Bruhat order on $\tbinom{[n]}{k}$ is given by $I = \{i_1< i_2 < \cdots < i_k\} \leq J = \{j_1<j_2 < \cdots <j_k\}$ if and only if $i_r \leq j_r$ for all $r$.  For $I \in \tbinom{[n]}{k}$, the Schubert matroid $\M_I$ is by definition the collection
$$
\M_I:=\{J \geq I \mid J \in \tbinom{[n]}{k}\}.
$$
It indexes the set of Pl\"ucker coordinates $\Delta_J$ that do not vanish on $X_I$.

Let $\I$ be a $(k,n)$-Grassmann necklace.  The \defn{positroid} $\M(\I)$ of $\I$ is the rank $k$ matroid on $n$ elements given by 
\begin{equation}\label{eq:positroid}
\M(\I) := \M_{I_1} \cap \chi(\M_{\chi^{-1}(I_2)}) \cap \cdots \cap \chi^{n-1}(\M_{\chi^{1-n}(I_n)}).
\end{equation}
If $\I = \I(f)$ then we write $\M(f)$ for $\M(\I)$.

\subsection{Positroid varieties}
%

Let the generator $\chi$ of the cyclic group $\Z/n\Z$ act on $[n]$ by the formula $\chi(i) = i+1 \mod n$. 
Then $\chi$ also acts on subsets of $[n]$.  

Define the \defn{positroid variety} $\Pi_f$ 
\begin{equation}\label{eq:posdef}
\Pi_f = \Pi_\I := \X_{I_1} \cap \chi(\X_{\chi^{-1}(I_2)}) \cap \cdots \cap \chi^{n-1}(\X_{\chi^{1-n}(I_n)})
\end{equation}
where $\I= \I(f)$ and the \defn{open positroid variety} $\oPi_f$
$$
\oPi_f = \oPi_\I := \oX_{I_1} \cap \chi(\oX_{\chi^{-1}(I_2)}) \cap \cdots \cap \chi^{n-1}(\oX_{\chi^{1-n}(I_n)}).
$$
%
%
%

By \cite{KLS,KLS2}, the restriction map $\Gamma(\Gr(k,n),\O(d)) \to \Gamma(\Pi_f,\O(d))$ is surjective, where $\O(1)$ is the line bundle on $\Gr(k,n)$ associated to the Pl\"ucker embedding (and in particular, $\Gamma(\Pi_f,\O(d)) = 0$ for $d < 0$).  Thus the homogeneous coordinate ring $R(\Pi_f) := \bigoplus_{d \geq 0} \Gamma(\Pi_f,\O(d))$ of $\Pi_f$ is a quotient of the homogenous coordinate ring $R(k,n)$.  We write $\I(\Pi_f)$ for the homogeneous ideal of $\Pi_f$ and denote by $\hPi_f := \Spec(R(\Pi_f)) \subset \hGr(k,n)$ the affine cone over the positroid variety $\Pi_f$.

Recall that for $X \in \Gr(k,n)$, the matroid $\M_X$ of $X$ is defined as
$$
\M_{X}:= \{J \in \tbinom{[n]}{k} \mid \Delta_J(X) \neq 0\}.
$$
Define $\Pi_{f,>0}:= \oPi_f \cap \Gr(k,n)_{\geq 0}$.  The following result of Oh characterizes the matroids of totally nonnegative points.
\begin{theorem}[\cite{Oh}] \label{thm:Oh}
For any $X \in \Pi_{f,>0}$, we have $\M_X = \M(f)$.
\end{theorem}

\section{The cyclic Demazure module}
\label{sec:cyclicDem}
\subsection{Demazure modules and Demazure crystals}
%
Let $\I(X_I) \subset R(k,n)$ denote the homogeneous ideal of the Schubert variety $X_I$ (see Section \ref{sec:Schub}) and let $\I(X_I)_d \subset R(k,n)_d$ denote the degree $d$ component.  Let $R(X_I)_d = \Gamma(X_I,\O(d))$ denote the degree $d$ part of the homogeneous coordinate ring of $X_I$.  The restriction map $\Gamma(\Gr(k,n),\O(d)) \to \Gamma(X_I,\O(d))$ is known to be surjective, and thus the space $R(X_I)_d$ is naturally a quotient of $R(k,n)_d=V(d\omega_k)^*$.

For $I \in \tbinom{[n]}{k}$, we have an extremal weight vector $G(T_I) \in V(d\omega_k)$.  The vector $G(T_I)$ spans the weight space of $V(d\omega_k)$ with weight $\alpha = (\alpha_1,\ldots,\alpha_n)$ given by $\alpha_i = d$ if $i \in I$ and $\alpha_i = 0$ otherwise.  The \defn{Demazure module} $V_I(d\omega_k)$ is defined to be the $B_-$-submodule of $V(d\omega_k)$ generated by the vector $G(T_I)$.  It is a classical result that $\I(X_I)_d$ can be identifed with $ V_{I}(d\omega_k)^\perp \subset V(d\omega_k)^*$ (see for example \cite[Chapter 8]{Kum}).

\medskip

For $I \in \tbinom{[n]}{k}$, we have a tableau $T_I \in B(d\omega_k)$ with all entries in the $r$-th row equal to $i_r$, where $I = \{i_1,i_2,\ldots,i_k\}$.  The canonical basis vector $G(T_I)$ is an extremal weight vector of $V(d\omega_k)$.  Define the \defn{Demazure crystal} $B_I(d\omega_k) \subset B(d\omega_k)$ to be the subset of $B(d\omega_k)$ obtained by repeatedly applying the operators $\tf_1,\tf_2,\ldots,\tf_{n-1}$ to $T_I$.

The following result is due to Kashiwara \cite{Kasdem}.
\begin{theorem}\label{thm:Kas}
The $B_-$-submodule $V_I(d\omega_k)$ has a basis $\{G(T) \mid T \in B_I(d\omega_k)\}$.
\end{theorem}

By Theorem \ref{thm:Kas}, we obtain:
\begin{proposition}\label{prop:SchubDemazure} \
We have 
\begin{enumerate}
\item
$\I(X_I)_d = V_{I}(d\omega_k)^\perp \subset V(d\omega_k)^* = R(k,n)_d$ has a basis given by $\{\H(T)\mid T \notin B_I(d\omega_k)\}$.
\item
$R(X_I)_d$ has a basis given by (the image of) $\{\H(T) \mid T \in B_I(d\omega_k)\}$.
\end{enumerate}
\end{proposition}

Let us give a more explicit description of $B_I(d\omega_k)$.

\begin{proposition}
The set $B_I(d\omega_k)$ consists of tableaux $T$ which are entry-wise greater than or equal to $T_I$.
\end{proposition}
\begin{proof}
Let $S$ denote the set of tableaux $T \in B(d\omega_k)$ that are entry-wise greater than or equal to $T_I$.
Since the operators $\tf_i$ decreases a single entry of a tableau, it is clear that $B_I(d\omega_k)$ is contained in $S$.  Also, it is known that the set $S$ indexes a basis for $R(X_I)_d$ known as the standard monomial basis, see for example \cite{LaLi}.  Thus $|B_I(d\omega_k)| = |S|$, so $B_I(d\omega_k) = S$.
\end{proof}

\subsection{Cyclic Demazure modules}
Let $f$ be a $(k,n)$-bounded affine permutation.  Define $\I(\Pi_f)_d \subset R(k,n)_d$ by
$$
\I(\Pi_f)_d\coloneqq \I(\Pi_f) \cap R(k,n)_d
$$
to be the degree $d$ homogeneous component of $\I(\Pi_f)$.  Since $\I(\Pi_f)$ is a homogeneous ideal, it is spanned by the subspaces $\I(\Pi_f)_d$.  The aim of this section is to give a representation-theoretic description of $\I(\Pi_f)_d$ as a subspace of $R(k,n)_d \simeq V(d\omega_k)^*$.

Let $f \in \Bound(k,n)$ have $(k,n)$-Grassmann-necklace $\I(f) = (I_1,I_2,\ldots,I_n)$.  Define the \defn{cyclic Demazure crystal} $B_f(d\omega_k)$ to be intersection 
$$
B_f(d\omega_k)\coloneqq B_{I_1}(d\omega_k) \cap \chi(B_{\chi^{-1}(I_2)}(d\omega_k)) \cap \cdots \cap \chi^{n-1}(B_{\chi^{1-n}(I_n)}(d\omega_k)).
$$
If we identify $B(\omega_k)$ with the set $\tbinom{[n]}{k}$ of $k$-element subsets of $[n]$, then $B_f(\omega_k)$ is simply the positroid $\M(f)$ \eqref{eq:positroid}.  Also, define the \defn{cyclic Demazure module} $V_f(d\omega_k)$ to be intersection
\begin{equation}\label{eq:CDMdef}
V_f(d\omega_k) \coloneqq V_{I_1}(d\omega_k) \cap \chi(V_{\chi^{-1}(I_2)}(d\omega_k)) \cap \cdots \cap \chi^{n-1}(V_{\chi^{1-n}(I_n)}(d\omega_k)).
\end{equation}

Let $R(\Pi_f)$ denote the homogeneous coordinate ring of the positroid variety $\Pi_f$.  

\begin{theorem}\label{thm:canspan}
The subspace $V_f(d\omega_k)$ has a basis $\{G(T) \mid T \in B_f(d\omega_k)\}$.
\end{theorem}

\begin{theorem}\label{thm:Posideal}\
\begin{enumerate}
\item
$\I(\Pi_f)_d$ is isomorphic to $V_f(d\omega_k)^\perp$ and has a basis given by $\{\H(T) \mid T \notin B_f(d\omega_k)\}$.
\item
$R(\Pi_f)_d$ has a basis given by the images of $\{\H(T) \mid T \in B_f(d\omega_k)\}$.
\end{enumerate}
\end{theorem}
Theorem \ref{thm:Posideal} reduces in the case $d = 1$ to Theorem \ref{thm:Oh}.
\begin{cor}
For $f \in\Bound(k,n)$ and $d \geq 1$, the cyclic Demazure crystal $B_f(d\omega_k)$ is nonempty.
\end{cor}
Indeed, for any $a$, the function $\Delta_{I_a}$ is non-zero on $\hPi_f$, and so is $\Delta_{I_a}^d$.  Thus, $B_f(d\omega_k)$ contains the tableaux $T_{I_1}, T_{I_2},\ldots,T_{I_n}$.

\begin{remark}
It follows from Theorem \ref{thm:Posideal} that the vectors $\H(T) \in V(d\omega_k)^*$ that do not restrict to identically zero on $\Pi_f$ form a basis for $R(\Pi_f)_d$.  This is not the case for the standard monomial basis (cf. \cite{LaLi}).
\end{remark}

\begin{example}
Suppose $k = 1$.  In this case $B_I(d\omega_1)$ is the set of one-row tableaux (of length $d$) with entries in $1,2,\ldots,i$, where $I = \{i\}$.  By choosing the $(1,n)$-Grassmann necklace appropriately, $B_f(\omega_1)$ can be arranged to be any subset of $\{1,2,\ldots,n\}$.  For example, if $n = 4$, $(I_1,I_2,I_3,I_4) = (1,3,3,1)$ gives $B_f(\omega_1) = \{1,3\}$.  The set $B_f(d\omega_1)$ is simply the set of one-row tableaux with entries in $B_f(\omega_1)$.
\end{example}

\begin{example}\label{ex:cyclicDem}
Take $k = 2$ and $n = 4$.  Let us consider the positroid variety $\Pi_f$ where $f = [2547] \in \B(2,4)$.  The Grassmann necklace is $\I(f) = (13,23,13,41)$.  The set $B_f(2\omega_2)$ is given by the set of tableaux
$$
\tableau[sbY]{
1&1\\
3&3
}
\;\;
\tableau[sbY]{
1&2\\
3&3
}
\;\;
\tableau[sbY]{
1&2\\
3&4
}
\;\;
\tableau[sbY]{
1&1\\
3&4
}
\;\;
\tableau[sbY]{
2&2\\
3&3
}
\;\;
\tableau[sbY]{
2&2\\
3&4
}
\;\;
\tableau[sbY]{
1&1\\
4&4
}
\;\;
\tableau[sbY]{
1&2\\
4&4
}
\;\;
\tableau[sbY]{
2&2\\
4&4
}
$$
\end{example}

\begin{example}\label{ex:2}
Consider $k = 2$ and $n = 5$.  Let us consider the positroid variety $\Pi_f$ where $f = [63547] \in \B(2,5)$ and compute $B_f(2\omega_2)$.  The Grassmann necklace is $\I(f) = (12,12,13,15,15)$.  Since $B_{12}(2\omega_2) = B(2\omega_2)$, we have
$$
B_f(2\omega_2) = \chi(B_{15}(2\omega_2)) \cap  \chi^2(B_{14}(2\omega_2)) \cap \chi^3(B_{23}(2\omega_2)).
$$
The set $B_{15}(2\omega_2)$ consists of all tableaux of the form $\tableau[sbY]{
a&b\\
5&5
}$ with $1 \leq a \leq b \leq 4$ and thus $\chi(B_{15}(2\omega_2))$ consists of all tableaux of the form $\tableau[sbY]{
1&1\\
a&b
}$ with $2 \leq a \leq b \leq 5$.  In particular, every tableau in $B_f(2\omega_2)$ has exactly two 1-s.  Intersecting with $\chi^2(B_{14}(2\omega_2))$ imposes no additional restriction.  On the other hand, looking at tableaux in $B_{23}(2\omega_2)$ with two 3-s, we get the six tableaux
$$
\tableau[sbY]{
2&2\\
3&3
}
\;\;
\tableau[sbY]{
2&3\\
3&4
}
\;\;
\tableau[sbY]{
2&3\\
3&5
}
\;\;
\tableau[sbY]{
3&3\\
4&4
}\;\;
\tableau[sbY]{
3&3\\
4&5
}
\;\;
\tableau[sbY]{
3&3\\
5&5
}
$$
and thus $B_f(2\omega_2)$ consists of the tableaux
$$
\tableau[sbY]{
1&1\\
5&5
}
\;\;
\tableau[sbY]{
1&1\\
2&5
}
\;\;
\tableau[sbY]{
1&1\\
3&5
}
\;\;
\tableau[sbY]{
1&1\\
2&2
}\;\;
\tableau[sbY]{
1&1\\
2&3
}
\;\;
\tableau[sbY]{
1&1\\
3&3
}
$$
\end{example}

We give an example of a Schubert variety whose ideal does not have a basis given by a subset of the dual canonical basis.\begin{example}\label{ex:13}
Let $X \subset \Gr(2,4)$ be given by the single equation $\{\Delta_{13} = 0\}$.  This is a permutation of a standard Schubert variety that is not a positroid variety.  Then the degree two part of $\I(X)$ has a one-dimensional weight space for the weight $(1,1,1,1)$.  It is spanned by the vector $\Delta_{13}\Delta_{24}$.  This vector is a sum of two elements of the dual canonical basis by Theorem \ref{thm:LamTL}.
\end{example}

%
%

\subsection{Proof of Theorem \ref{thm:canspan}} 
By Theorem \ref{thm:Kas}, $V_{\chi^{1-a}(I_a)}(d\omega_k)$ has basis $\{G(T) \mid T \in B_{\chi^{1-a}(I_a)}(d\omega_k)\}$.  Thus by Theorem \ref{thm:canbasis}(4), the rotation $\chi^{a-1}(V_{\chi^{1-a}(I_a)}(d\omega_k))$ has basis given by 
$$
\{G(T) \mid T \in \chi^{a-1}B_{\chi^{1-a}(I_a)}(d\omega_k)\}.
$$
It follows that the intersection \eqref{eq:CDMdef} has basis $\{G(T) \mid T \in B_f(d\omega_k)\}$, establishing Theorem \ref{thm:canspan}.

\subsection{Proof of Theorem \ref{thm:Posideal}}
Our proof of Theorem \ref{thm:Posideal} relies on the following result proved jointly with Knutson and Speyer~\cite{KLS,KLS2}.  It states that the intersection in \eqref{eq:posdef} is reduced, so the equality in \eqref{eq:posdef} holds as schemes.
\begin{proposition}\label{prop:posideal}
The homogeneous ideal $\I_f$ of a positroid variety is given by 
$$
\I(\Pi_f) = \I(X_{I_1}) + \chi( \I(X_{\chi^{-1}(I_2)}))+\cdots + \chi^{n-1}( \I(X_{\chi^{1-n}(I_n)})).
$$
\end{proposition}

Since $(A \cap B)^\perp = A^\perp + B^\perp$, combining with Proposition \ref{prop:SchubDemazure}, we deduce that
$$
\I(\Pi_f)_d =\left( V_{I_1}(d\omega_k) \cap \chi(V_{\chi^{-1}(I_2)}(d\omega_k)) \cap \cdots \cap \chi^{n-1}(V_{\chi^{1-n}(I_n)}(d\omega_k)) \right)^\perp = V_f(d\omega_k)^\perp.
$$
Combining with Theorem \ref{thm:Kas}, we obtain Theorem \ref{thm:Posideal}(1).  Theorem \ref{thm:Posideal}(2) follows from the isomorphism of vector spaces $R(\Pi_f)_d \cong R(k,n)_d/\I(\Pi_f)_d$.


\subsection{Positivity}

\begin{theorem}\label{thm:poscan} \
For $f \in \Bound(k,n)$ and $T \in B(d\omega_k)$, if $\H(T)$ is not identically zero on $\Pi_f$, then it takes strictly positive values everywhere on $\Pi_{f,>0}$.
\end{theorem}
\begin{proof}
Fix $f \in \Bound(k,n)$.  By \cite[Section 7]{LamCDM}, the totally nonnegative cell $\Pi_{f,>0} \cong \R_{>0}^d$ has a parametrization of the following form:
$$
\R_{>0}^d \ni (t_1,t_2,\ldots,t_d) \mapsto x_{i_1}(t_1) \cdots x_{i_d}(t_d) \cdot e_I \in \Pi_{f,>0}
$$
where $e_I \in \Gr(k,n)_{\geq 0}$ for $I = \{i_1,\ldots,i_k\} \in \tbinom{[n]}{k}$ denotes the point 
$$
e_I = \spn(e_{i_1},e_{i_2},\ldots,e_{i_k}) \in \Gr(k,n)_{\geq 0}
$$ and $x_i(a) = \exp(a e_{i,i+1}) \in \GL(n)$ is the one parameter subgroup associated to the Chevalley generator for $i = 1,2,\ldots,n-1$.  For $i = 0$, we define $x_0(a)$ by conjugating $x_1(a)$ by $\chi$.

Fix a lift of $e_I$ to $\hGr(k,n)$.  Then the value of the dual canonical basis element $H(T)\in V(d\omega_k)^*$ on the point $X = x_{i_1}(t_1) \cdots x_{i_d}(t_d) \cdot e_I \in \hGr(k,n)$ is given by
$$
\langle H(T), x_{i_1}(t_1) \cdots x_{i_d}(t_d) \cdot G(T_I) \rangle
$$
where $\langle \cdot, \cdot \rangle$ denotes the natural pairing between $V(d\omega_k)^*$ and $V(d\omega_k)$, and $G(T_I)$ is the canonical basis element of extremal weight indexed by $T_I$.  

Thus it suffices to show that for any $T \in B(d\omega_k)$, the coefficient of $G(T)$ in $x_{i_1}(t_1) \cdots x_{i_d}(t_d) \cdot G(T_I)$ is equal to a (possibly zero) polynomial in $t_1,t_2,\ldots, t_d$ with nonnegative coefficients.  For $i = 1,2,\ldots,n-1$, it follows from the proof of \cite[Proposition 3.2]{LusTP} that the matrix coefficients of $x_i(a)$ on the canonical basis of $V(d\omega_k)$ are polynomials in $t_1,t_2,\ldots, t_d$ with nonnegative coefficients.  By Theorem \ref{thm:canbasis}(4), the same holds for $i = 0$.  The claim follows.
\end{proof}

%


%

\section{Future directions}\label{sec:future}
\subsection{The character of the cyclic Demazure module}
The character of highest weight representation $V(d\omega_k)$ is given by the celebrated Weyl character formula.  The character of the Demazure module $V_I(d\omega_k)$ is given by the Demazure character formula \cite{Dem,And}.

\begin{problem}
Find a formula for the character of $V_f(d\omega_k)$.  Equivalently, compute the weight generating function of $B_f(d\omega_k)$.
\end{problem}

For the bounded affine permutation $f = [2547]$ of Example \ref{ex:cyclicDem}, we have
$$
{\rm ch}(V_f(2\omega_2)) = x_1^2x_3^2 + x_1^2x_3x_4 + x_1^2 x_4^2+ x_1x_2x_3^2 + x_1x_2x_3x_4+ x_1x_2x_4^2 + x_2^2 x_3^2 + x_2^2x_3x_4 + x_2^2 x_4^2
$$
and for $f = [63547]$ of Example \ref{ex:2}, we have
$$
{\rm ch}(V_f(2\omega_2)) = x_1^2(x_2^2 + x_2x_3 + x_2x_5 + x_3^2 + x_3x_5+x_5^2).
$$

\subsection{Quantization}

Quantum versions of Grassmannians and Schubert varieties have been studied by many authors, see for example \cite{LR}.  In that setting, positroid varieties correspond to certain torus-invariant prime ideals, classified in \cite{MC,Y}.  

\begin{problem}
Find the quantum version of the cyclic Demazure module, and quantum versions of Theorems~\ref{thm:canspan} and \ref{thm:Posideal}.
\end{problem}
Note however that the cyclic symmetry acts on the quantum Grassmannian in a more subtle way than it does on the Grassmannian \cite{LLtwist}.

\subsection{Higher degree matroids}
\begin{definition}
For an integer $d \geq 1$, and $X \in \Gr(k,n)$, define the \defn{degree $d$ canonical basis matroid}
$$
\M_{X,d}:= \{T \in B(d\omega_k) \mid \H(T)(X) \neq 0\}.
$$
\end{definition}

By Theorem \ref{thm:canspan}(1), for $d = 1$, $\M_{X,1}$ is the usual matroid of $X$.  By Theorem \ref{thm:poscan}, the degree $d$ canonical basis matroids $\M_{X,d}$ of a point $X \in \Gr(k,n)_{\geq 0}$ is completely determined by the usual positroid $\M_X$.  Thus for any $d$, there is a natural bijection between $\Bound(k,n)$ and the set of degree $d$ positroids, sending $f \in \Bound(k,n)$ to $B_f(d\omega_k)$.

\begin{problem}
Find axioms for degree $d$ canonical basis matroids.
\end{problem}

\subsection{Projective geometry interpretation of dual canonical basis}
Let $X$ be a $k \times n$ matrix representing a point in $\Gr(k,n)$.  We assume that all columns of $X$ are non-zero and think of $X$ as a collection $p_1,p_2,\ldots,p_n$ of points in $\P^{k-1}$.  The vanishing of $\Delta_{i_1,\ldots,i_k}(X)$ is equivalent to the geometric statement that $p_{i_1},\ldots,p_{i_k}$ do not span the whole of $\P^{k-1}$.

The Temperley-Lieb invariants $\Delta_{(\tau,T)}$ can be interpreted as tensor invariants \cite[Appendix]{FLL}, and thereby we obtain an interpretation of degree two matroids in geometric terms.  For example, let $k = 3$ and $n = 6$.  If $\tau = \{(1,6),(2,5),(3,4)\}$, then $\Delta_{(\tau,\emptyset)} = \Delta_{123}\Delta_{456}$, which vanishes if and only if either $p_1,p_2,p_3$ are colinear or $p_4,p_5,p_6$ are colinear.  If $\tau = \{(1,2),(3,4),(5,6)\}$, then $\Delta_{(\tau,\emptyset)}$ vanishes if the lines $\overline{23}, \overline{45}$, and $\overline{16}$ have a common intersection point.   (If any of these pairs, say $p_2$ and $p_3$, do not span a line, then $\Delta_{(\tau,\emptyset)}$ also vanishes.)

%

\begin{problem}
Give an interpretation of the vanishing of the function $H(T)$, $T \in B(d\omega_k)$ in projective geometry terms.
\end{problem}
See also \cite{FP,BHL,LamCDM} for related work.

\subsection{Two maps on $R(k,n)$}

We have a map 
$$
\wedge: \Lambda^\ell(\C^n) \otimes \Lambda^k(\C^n) \to \Lambda^{k +\ell}(\C^n)
$$
induced by 
$$
(e_{i_1} \wedge e_{i_2} \wedge \cdots \wedge e_{i_\ell}) \otimes (e_{j_1} \wedge e_{j_2} \wedge \cdots \wedge e_{j_k}) \longmapsto e_{i_1} \wedge \cdots \wedge e_{i_\ell} \wedge e_{j_1} \wedge \cdots \wedge e_{j_k}.
$$
This is a map of $\GL(n)$-representations, and up to scalar, it is the unique such map, since the multiplicity of $\Lambda^{k +\ell}(\C^n)$ in $\Lambda^\ell(\C^n) \otimes \Lambda^k(\C^n)$ is equal to one.

Similarly, by \cite[Theorem 3.1]{Ste}, the $\GL(n)$-representation $V(d\omega_\ell) \otimes V(d\omega_k)$ is multiplicity free, and in particular, the irreducible representation $V(d\omega_{k+\ell})$ appears with multiplicity one. We thus have a canonical (up to scalar) surjective map
$$
\kappa^d_{k,\ell}: V(d\omega_\ell) \otimes V(d\omega_k) \rightarrow V(d\omega_{k+\ell})
$$
of $\GL(n)$-representations.

\begin{problem}\label{prob:sum}
Give an explicit combinatorial formula for the expansion of $\kappa^d_{k,\ell}(G(T) \otimes G(T'))$ in the canonical basis.
\end{problem}

We have obtained an explicit combinatorial solution to Problem \ref{prob:sum} for Temperley-Lieb invariants, which we hope to explain elsewhere.

\begin{remark}

Let $\phi_{k,\ell}: \Gr(k,n) \times \Gr(\ell,n) \dashrightarrow \Gr(k+\ell,n)$ be the (rational) direct sum map \cite{LamCDM}, sending $(V,W)$ to $V \bigoplus W$.
It is not difficult to see that the map $\kappa^d_{k,\ell}$ can be chosen to be dual to the map $R(k+\ell,n)_d \to R(\ell,n)_d \otimes R(k,n)_d $ induced by $\phi_{k,\ell}$.  The study of $\phi_{k,\ell}$ was one of the main motivations for the present work, and we refer the reader to \cite{LamCDM} for further details.
\end{remark}

\begin{remark}
An alternating formula for $\kappa^d_{k,\ell}(G(T) \otimes G(T'))$ can be computed in terms of Kazhdan-Lusztig polynomials, for example by work of Brundan \cite{Bru}.
\end{remark}

Similarly, there is (up to scalar) a unique non-trival $\GL(n)$ homomorphism
$$
\eta^{d,d'}_{k}: V((d+d')\omega_{k}) \rightarrow  V(d\omega_\ell) \otimes V(d'\omega_k).
$$
This map is dual to the natural multiplication map $R(k,n)_d \otimes R(k,n)_{d'} \to R(k,n)_{d+d'}$ of the homogeneous coordinate ring.

\begin{problem}
Give an explicit combinatorial formula for the expansion of $\eta^{d,d'}_{k}(G(T))$ in the canonical basis.
\end{problem}

\end{document}